\documentclass[12pt]{article}

\usepackage[margin=1in]{geometry}  % set the margins to 1in on all sides
\usepackage{graphicx}              % to include figures
\usepackage{amsmath}               % great math stuff
\usepackage{amsfonts}              % for blackboard bold, etc
\usepackage{amsthm}                % better theorem environments
\usepackage{enumerate}
\usepackage{mathtools}
\usepackage{multicol}
\usepackage{booktabs}
\usepackage{caption}
\usepackage{color}

\renewcommand{\bar}[1]{\overline{#1}}    % Overline looks nicer for conjugation...
\theoremstyle{plain}
\newtheorem{thm}{Theorem}[section] % Kai
\newtheorem{lem}[thm]{Lemma} % Kai
\newtheorem{prop}[thm]{Proposition} % Kai
\theoremstyle{definition}
\newtheorem{definition}[thm]{Definition}

\theoremstyle{remark}

\newcommand{\normalarray}{\renewcommand{\arraystretch}{1.1}}

\normalarray

\newcommand{\RR}{\mathbb{R}}      % for Real numbers
\newcommand{\CC}{\mathbb{C}}      % for Complex numbers
\newcommand{\Jq}{\J^{(q)}}
\newcommand{\Aq}{\A^{(q)}}

\newcommand{\Cq}{\C^{(q)}}
\newcommand{\Dq}{\D^{(q)}}

\newcommand{\A}{\mathcal{A}}

\newcommand{\C}{\mathcal{C}}
\newcommand{\D}{\mathcal{D}}

\newcommand{\J}{\mathcal{J}}

\newcommand{\X}{\mathcal{X}}
\newcommand{\Y}{\mathcal{Y}}

\newcommand{\GFq}{\text{GF}(q)}

\newcommand{\BHnk}{\text{BH}(n, k)}
\newcommand{\BH}[2]{\text{BH}(#1, #2)}

\newcommand{\FUHnm}{\text{QUH}(n, m)}
\newcommand{\FUH}[2]{\text{QUH}(#1, #2)}

 {\begin{list}{}%
    {\setlength{\leftmargin}{#1}}%
  \item[]%
 }
{\end{list}}

\title{On a class of quaternary complex Hadamard matrices} 
\date{
\today
}

\author{
Kai Fender\thanks{Department of Mathematics and Computer Science, University of Lethbridge,
Lethbridge, Alberta, T1K 3M4, Canada.  \texttt{k.fender@uleth.ca}} \and Hadi Kharaghani\thanks{Department of Mathematics and Computer Science, University of Lethbridge,
Lethbridge, Alberta, T1K 3M4, Canada. \texttt{kharaghani@uleth.ca}}
\and  
 Sho Suda\thanks{Department of Mathematics Education,  Aichi University of Education, 1 Hirosawa, Igaya-cho,  Kariya, Aichi, 448-8542, Japan. \texttt{suda@auecc.aichi-edu.ac.jp}}
}

\begin{document}
\maketitle
\begin{abstract}
We introduce a class of regular unit Hadamard matrices whose entries consist of two complex numbers and their conjugates for a total of four complex numbers. We then show that these matrices 
are contained in the Bose-Mesner algebra of an association scheme arising from skew Paley matrices.
\end{abstract}

\section{Introduction}
An $n \times n$ matrix $H$ is a \emph{unit Hadamard matrix} if its entries are all complex numbers of modulus 1 and it satisfies $HH^* = nI_n$. If the entries of $H$ are all complex $k^\text{th}$ roots of unity, it is called a \emph{Butson Hadamard matrix}, referred to as a $\BHnk$, and the particular case of $k=2$ is a \emph{Hadamard matrix}. Following Compton et al.~\cite{Compton2015} we call a Butson or unit Hadamard matrix \textit{unreal} if its entries are strictly in $\CC \setminus \RR$. A Hadamard matrix $H$ of order $n$ is called to be of \emph{skew type}, if  $H=I +W$, where $W$ is a skew symmetric $(0,\pm 1)$-matrix. It follows that $WW^T=(n-1)I_n$. For a thorough examination of unit and Butson Hadamard matrices, we refer the reader to Sz\"oll\H{o}si's PhD thesis~\cite{Szollosi}, and for some fundamental results and applications of Hadamard matrices, we refer the reader to Seberry and Yamada's 1992 survey~\cite{seberrySurvey}.
Given a matrix $A$ of order $n$, let $R_i$ denote the $i$-th row of $A$, $S(R_i)$ the sum of all entries of $R_i$ and $S(A)$, called the \emph{excess of $A$}, the sum of all its entries. 
A result of ~\cite{Best} implies that for a unit Hadamard matrix of order $n$, $|S(A)| \le n \sqrt n$ and equality occurs if and only if $|S(R_i)|=\sqrt n$ for $1\le i \le n$. A unit Hadamard matrix $A$ of order $n$ is called \emph{regular} if $|S(R_i)|=\sqrt n$ for $1\le i \le n$, see \cite{KS} for details.

In this paper we introduce a recursive method to construct pairs of $(\pm 1)$-matrices satisfying two specific equations. Similar recursive methods were presented in 2005 to obtain symmetric designs and orthogonal designs~\cite{recOrthogonalDesigns, recSymmetricDesigns}. 

Assuming the existence of a skew type Hadamard matrix of order $q+1$, we show the pairs of matrices obtained from our recursive method can be used to construct infinite classes of  a special type of unit Hadamard matrices of order $q^m$, for each positive integer $m$,  which we have dubbed \emph{quaternary unit Hadamard} matrices. In particular, as a corollary we conclude that for each prime power $q \equiv 3 \pmod{4}$, there are infinite classes of unreal $\BH{3^m}{6}$'s, and quaternary unit Hadamard matrices of order $q^m$. Moreover, we will demonstrate that all of the constructed Butson Hadamard matrices and quaternary unit Hadamard matrices are regular, and some of those have   multicirculant structure.

Some of the results in this paper are closely related to part of the results in a recent paper by Compton et al.~\cite{Compton2015}, see also \cite{muk}. Among other results, Compton et al. proved the existence of $\BH{3^m}{6}$'s for each integer $m \ge 0$. Herein, we too will construct $\BH{3^m}{6}$'s. However, our matrices are distinguished from those of Compton et al. in that our $\BH{3^m}{6}$'s are regular and multicirculant. In their paper, Compton et al. also showed that a $\BH{n}{6}$ is equivalent to a pair of amicable $(\pm 1)$-matrices satisfying a certain equation, and that this pair of matrices can be used to construct a Hadamard matrix. We have generalized this result in Section~\ref{sect:concepts} by introducing quaternary Hadamard matrices and showing that they are equivalent to a pair of amicable $(\pm 1)$-matrices satisfying an equation analogous to that introduced by Compton et al. Moreover, we will show that the pairs of amicable $(\pm 1)$-matrices equivalent to quaternary unit Hadamard matrices can be used to construct Hadamard matrices. Next, in Section~\ref{sect:result} we will introduce a recursive method to construct such pairs of matrices, and we will use this method to show that for each prime power $q \equiv 3 \pmod{4}$ and integer $m \ge 0$, we can construct infinite classes of  unreal $\BH{3^m}{6}$'s and unreal quaternary unit Hadamard matrices of order $q^m$.  Finally, in Chapter 4, we introduce an association scheme whose Bose-Mesner algebra contains our quaternary unit Hadamard matrices. 

\section{Quaternary Unit Hadamard Matrices}\label{sect:concepts}

\begin{definition}
\label{def:quaternary}
We say that an $n \times n$ unit Hadamard matrix $H$ is \textit{quaternary} if there is a positive integer $m$ such that the entries of $H$ are all in the set $\left\{\pm \frac{1}{\sqrt{m+1}} \pm i \sqrt{\frac{m}{m+1}}, \pm \frac{1}{\sqrt{m+1}} \mp i \sqrt{\frac{m}{m+1}}\right\}$. For short, we refer to such a quaternary unit Hadamard matrix as a $\FUHnm$. 

%CHANGE ALL $m$ TO $a$ throughout this section and the first paragraph of th next!!!!
\end{definition}

%\begin{rem}
%We could have equally well defined quaternary unit Hadamard matrices to have entries strictly in the set $\left\{ \pm \sqrt{\frac{m}{m+1}} \pm i \frac{1}{\sqrt{m+1}}, \pm \sqrt{\frac{m}{m+1}} \mp i \frac{1}{\sqrt{m+1}}\right\}$. Using either definition does not change any of the results pertaining to quaternary unit Hadamard matrices presented herein, and the proofs of the results are nearly identical regardless which definition is chosen.
%\end{rem}

It is readily verified that any $\FUH{n}{1}$ or $\FUH{n}{3}$ is also a Butson Hadamard matrix. 
%It turns out that if $m \ne 1$ and $m \ne 3$, then a $\FUHnm$ cannot be a Butson Hadamard matrix. The next lemma will help us prove this claim.

\begin{lem}
\label{lem:rootsOfUnity}
Let $m$ be a positive integer. Then $\zeta=\frac{1}{\sqrt{m+1}} + i \sqrt{\frac{m}{m+1}}$ is a root of unity if and only if $m=1$ or $m=3$.
\end{lem}

\begin{proof}
If $\zeta$ is a root of unity, then so are $\zeta^2$ and $\bar{\zeta}^2$. Thus $\zeta^2+{\bar{\zeta}}^2=\frac{-2(m-1)}{m+1}$ is an algebraic integer, and hence an integer. This implies that $m=1$ or 3. \end{proof}

%Let $\zeta = \frac{1}{\sqrt{m+1}} + i \frac{\sqrt{m}}{\sqrt{m+1}}$ and assume without loss of generality that $\zeta$ is a primitive $k^{th}$ root of unity. 

%Notice that $\zeta$ is a root of the polynomial $x^4+2\frac{m-1}{m+1}x^2+1$. It follows that $\phi(k)=\left[ \mathbb{Q}(\zeta): \mathbb{Q} \right] \leq 4$, where $\phi$ denotes Euler's totient function. Therefore, we can conclude that $k \in \{1,2,3,4,5,6,8,10,12\}$. Examining these cases by hand and keeping in mind that $m>0$, we see the only allowable possibilities are $k \in \{3, 6, 8, 12\}$. Thus $m=1$ or $m=3$. The converse is easily seen to be true.
%\end{proof}

The next proposition follows immediately from the previous lemma and the observation that any $\FUH{n}{1}$ or $\FUH{n}{3}$ is also a Butson Hadamard matrix.

\begin{prop}
A $\FUHnm$ is a Butson Hadamard matrix if and only if $m=1$ or $m=3$.
\end{prop}

We now demonstrate that $\FUHnm$'s are equivalent to pairs of $n \times n$ matrices satisfying certain properties. First, however, recall a definition.

\begin{definition}
Two complex matrices $A$ and $B$ are called \textit{amicable} if $AB^* = BA^*$.
\end{definition}

In the reference~\cite{Compton2015}, Compton et al. establish the following result.

\begin{thm}[Compton et al.,~\cite{Compton2015}]
\label{thm:comptonEquivalence}
An unreal $\BH{n}{6}$ is equivalent to a pair of $n \times n$ amicable $(\pm 1)$-matrices $A$ and $B$ satisfying $AA^T+3BB^T = 4nI_n$.
\end{thm}

With little difficulty, this result can be generalized in the following manner. Assume $H$ is a $\FUHnm$.
Then we can write
$$ H = \frac{1}{\sqrt{m+1}} A + i \sqrt{\frac{m}{m+1}} B $$
for some $(\pm 1)$-matrices $A$ and $B$. Therefore,
\begin{align*}
nI_n &= \left(\frac{1}{\sqrt{m+1}} A + i \sqrt{\frac{m}{m+1}} B\right)\left(\frac{1}{\sqrt{m+1}} A + i \sqrt{\frac{m}{m+1}} B\right)^*,
\end{align*}
so
$$ n(m+1)I_n = AA^T + m BB^T + i \sqrt{m}(BA^T-AB^T). $$
Since the right-hand-side must be real, this proves the following generalization of Theorem~\ref{thm:comptonEquivalence}.

\begin{thm}
\label{thm:comptonGen}
A $\FUHnm$ is equivalent to a pair of $n \times n$ amicable $(\pm 1)$-matrices $A$ and $B$ satisfying $AA^T + mBB^T = (m+1)nI_n$.
\end{thm}

%Here $j_q$, $Q_q$, and $I_{q+1}^\prime$ are defined similarly to their definitions in Equation~(\ref{eqn:order12Had}). We will continue to use these definitions of $j_q$, $Q_q$, and $I_{q+1}^\prime$ throughout the remainder of the paper.

\section{A Recursive Method}\label{sect:result}
%In Section~\ref{sect:concepts} we saw that pairs of $n \times n$ amicable $(\pm 1)$-matrices $A$ and $B$ satisfying $AA^T + mBB^T=(m+1)nI_n$, where $m$ is some positive integer, imply the existence of quaternary unit Hadamard matrices and Hadamard matrices. 
In this section we will introduce a recursive construction for pairs of matrices satisfying the aforementioned properties. 
We use  $j_n$ and $J_n$  to denote the $1 \times n$ and $n \times n$ all-ones matrices respectively. Subscripts will be dropped where no ambiguity arises.

Let  $q+1$ be the order of a skew type Hadamard matrix $H$. Multiply rows and columns of $H$, if necessary, to get the matrix  \[ \left (\begin{matrix} 1 & j\\-j^T & I+Q\end{matrix}\right).\] The $(0,\pm 1)$-matrix $Q=(q_{ij})_{i,j=1}^{q}$, called the \emph{skew symmetric core} of the skew type Hadamard matrix, is
skew symmetric,   $J_qQ=QJ_q=0$, and $QQ^T=qI_q-J_q$. For any odd prime power $q$ the 
Jacobsthal matrix of order $q$ defined by 
$$
q_{ij} = \chi_q(a_i - a_j)  (a_i,a_j\in GF(q))
$$
where $\chi_q$ denotes the quadratic character in $\GFq$, enjoys the following important properties:%. Jacobsthal matrices have three important properties that are the roots of Paley's construction.
\begin{enumerate}
\item $Q$ is symmetric if $q \equiv 1 \pmod{4}$ and  skew symmetric if $q \equiv 3 \pmod{4}$.
\item $J_qQ=QJ_q=0.$
\item $QQ^T = qI_q-J_q$.
\end{enumerate}
So, Jacobsthal matrices provide many examples of skew symmetric cores.

Let $q$ be the order of a skew symmetric core $Q$. Define the following matrices recursively for each nonnegative integer $m$.

\begin{align}
\label{eqn:JA}
\Jq_m & = \begin{cases} 
J_1 & \text{if } m=0\\
J_q \otimes \mathcal{A}^{(q)}_{m-1} & \text{otherwise}
\end{cases}    , &    
\Aq_m & = \begin{cases} 
J_1 & \text{if } m=0\\
I_q \otimes \mathcal{J}^{(q)}_{m-1} + Q \otimes \mathcal{A}^{(q)}_{m-1} & \text{otherwise.}
\end{cases}   
\end{align}
It should be noted that when no ambiguity arises, for brevity we will drop the superscripts on $\Jq_m$ and $\Aq_m$.

It is not hard to prove by induction that $\Jq_m$ and $\Aq_m$ are amicable for each  nonnegative integer $m$. Indeed, the base case is clear, and using the induction hypothesis together with the fact that $J_qQ=QJ_q=0$, note that

\begin{align*}
	\J_{m+1} \A_{m+1}^T &= (J_q \otimes \A_m)(I_q \otimes \J_m + Q \otimes \A_m)^T \\
	&= J_q \otimes (\A_m\J_m^T)  \\
	&= J_q \otimes (\J_m\A_m^T) \\
	&= (I_q \otimes \J_m + Q \otimes \A_m)(J_q \otimes \A_m)^T \\
	&= \A_{m+1}\J_{m+1}^T.
\end{align*}
It follows that $\Jq_m$ and $\Aq_m$ are amicable for each integer $m \ge 0$.
It is also straightforward to prove by induction that 
$$\Jq_m(\Jq_m)^T + q \Aq_m(\Aq_m)^T = q^m(q+1)I_{q^m}$$
whenever $q+1$ is the order of a skew type Hadamard matrix. Again the base case is clear. Using the induction hypothesis together with the facts that $Q$ is skew symmetric, that $QQ^T = qI_q - J_q$, and that $\Jq_m$ and $\Aq_m$ are amicable, we obtain
\begin{align*}
\J_{m+1}&\J_{m+1}^T + q \A_{m+1}\A_{m+1}^T \\
&= (J_q \otimes \A_m)(J_q \otimes \A_m^T) + q(I_q \otimes \J_m + Q \otimes \A_m)(I_q \otimes \J_m^T + Q^T \otimes \A_m^T) \\
&= qJ_q \otimes \A_m\A_m^T + qI_q \otimes \J_m\J_m^T - q\, Q \otimes \J_m\A_m^T + q\, Q \otimes \A_m\J_m^T + q \,QQ^T \otimes \A_m\A_m^T \\
&=  qJ_q \otimes \A_m\A_m^T + qI_q \otimes \J_m\J_m^T + q(qI_q - J_q) \otimes \A_m\A_m^T \\
&= qI_q \otimes (\J_m\J_m^T + q \A_m\A_m^T) \\
& = q^{m+1}(q+1)I_{q^{m+1}}.
\end{align*}
Therefore, for each skew symmetric core of order $q$ and integer $m \ge 0$, the matrices $\Jq_m$ and $\Aq_m$ are a pair of amicable $(\pm 1)$-matrices satisfying 
$$\Jq_m(\Jq_m)^T + q \Aq_m(\Aq_m)^T = q^m(q+1)I_{q^m}.$$
Thus, using the results of Section~\ref{sect:concepts} we obtain a $\FUH{q^m}{q}$. Explicitly, the quaternary unit Hadamard matrices are
\begin{align}
\label{eqn:rec4UH}
\frac{1}{\sqrt{q+1}} \Jq_m + i \sqrt{\frac{q}{q+1}} \Aq_m.
\end{align}
%while the Hadamard matrices are
%\begin{align}
%\label{eqn:recHad}
%\left(
%  \begin{array}{cc}
%    0 & j_q \\
  %  j_q^T & Q_q \\
%  \end{array}
%\right)
%\otimes \Aq_m + I_{q+1}^\prime \otimes \Jq_m.
%\end{align}
In the next subsection, we will show that these quaternary unit Hadamard matrices  have some interesting properties. However, first we will take this opportunity to make a brief comment on our recursive method. Notice that our proof that $\Jq_m$ and $\Aq_m$ are a pair of amicable $(\pm 1)$-matrices satisfying 
$$\Jq_m(\Jq_m)^T + q \Aq_m(\Aq_m)^T = q^m(q+1)I_{q^m}$$
only relied on the fact that $J_1$ is amicable with itself and that $J_1J_1^T + qJ_1J_1^T = (q+1)I_1$. Therefore, it is straightforward to see that given any pair of $n \times n$ amicable $(\pm 1)$-matrices $X$ and $Y$ satisfying $XX^T + qYY^T=n(q+1)I_n$, where $q$ is the order of a skew symmetric core, the matrices
\begin{align*}
  \X_m & = \begin{cases} 
    X & \text{if } m=0\\
    J_q \otimes \Y_{m-1} & \text{otherwise}
  \end{cases}    , &
  \Y_m & = \begin{cases} 
    Y & \text{if } m=0\\
    I_q \otimes \X_{m-1} + Q \otimes \Y_{m-1} & \text{otherwise}
  \end{cases}   
\end{align*}
are amicable and satisfy
$$
\X_m \X_m^T + q\Y_m \Y_m^T = nq^m(q+1)I_{nq^m}
$$
for each integer $m \ge 0$.

As an application, for the prime power  $q \equiv 1 \pmod{4}$, by using the Jacobsthal matrix $Q$ and complex numbers we can  get a recursive construction similar to that in Equation~(\ref{eqn:JA}):
\begin{align}
\label{eqn:CD}
\Cq_m & = \begin{cases} 
J_1 & \text{if } m=0\\
J_q \otimes \Dq_{m-1} & \text{otherwise}
\end{cases}    , &    
\Dq_m & = \begin{cases} 
J_1 & \text{if } m=0\\
I_q \otimes \Cq_{m-1} + i \, Q \otimes \Dq_{m-1} & \text{otherwise}
\end{cases}.   
\end{align}
Almost identical proofs to those above show that $\Cq_m$ and $\Dq_m$ are always amicable and that
$$
\Cq_m(\Cq_m)^* + q\Dq_m(\Dq_m)^* = q^m(q+1)I_{q^m}.
$$
Therefore,  the following is a  unit Hadamard matrix (each entry of the matrix being one of $\pm 1, \pm i$, it is called a \emph{quaternary Hadamard matrix}) for each $m \ge 0$ and prime power $q \equiv 1 \pmod{4}$.
$$
\left(
  \begin{array}{cc}
    0 & j_q \\
    j_q^T & Q
  \end{array}
\right)
\otimes \Dq_m + i \, I_{q+1} \otimes \Cq_m.
$$

\subsection{An infinite class of quaternary unit Hadamard matrices}

By using the appropriate Jacobsthal matrix $Q$ in the definition of $\Jq_m$ and $\Aq_m$, we can ensure that the resulting $\FUH{q^m}{q}$'s have an interesting structure, which we have dubbed \textit{multicirculant}. Before defining this structure, we remind the reader that a circulant matrix with first row $(a_1, \dots, a_n)$ is denoted $\text{circ}(a_1, \dots, a_n)$, and that a block-circulant matrix is a matrix of the form $\text{circ}(A_1, \dots, A_n)$, where the $A_i$ are its blocks.

\begin{definition}
Let $M$ be a matrix of order $n$. If $n=1$, then we call $M$ a multicirculant matrix. If $n>1$, then we call $M$ multicirculant if and only if it is a block-circulant matrix whose blocks are multicirculant matrices.
\end{definition}

%Less formally, a multicirculant matrix is a block-circulant matrix whose blocks are block-circulant, whose blocks' blocks are in turn block circulant, whose blocks' blocks' blocks are also block-circulant, etc. 

Two facts about multicirculant matrices are straightforward to verify and will be used shortly. First, the Kronecker product of two multicirculant matrices is itself a multicirculant matrix. Second, if $A$ and $B$ are two multicirculant $n \times n$ matrices such that all their blocks are of the same dimensions, then $A+B$ is also a multicirculant matrix. 

For any odd prime power $q$, it is well known that one can construct a multicirculant Jacobsthal matrix $Q$. When $q \equiv 3 \pmod{4}$, use a multicirculant Jacobsthal matrix to construct $\Jq_m$ and $\Aq_m$. Then the two facts listed in the previous paragraph imply that $\Jq_m$ and $\Aq_m$ will be multicirculant. It follows that the $\FUH{q^m}{q}$'s in Equation~(\ref{eqn:rec4UH}) are multicirculant.

The $\FUH{q^m}{q}$'s in Equation~(\ref{eqn:rec4UH}) have another interesting property: they have maximal excess. To prove this, we introduce a lemma.

\begin{lem}
\label{lem:excessJA}
The following holds for all integers $m \ge 0$ and prime powers $q \equiv 3 \pmod{4}$.\end{lem}
\begin{enumerate}[(i)]
\item $S(\mathcal{J}^{(q)}_{2m}) = S(\mathcal{A}^{(q)}_{2m}) = q^{3m}$.
\item $S(\mathcal{J}^{(q)}_{2m+1}) = q^{3m+2}$.
\item $S(\mathcal{A}^{(q)}_{2m+1}) = q^{3m+1}$.
\end{enumerate}
%\end{lem}

\begin{proof}
First notice that
\begin{equation}
\begin{split}
 S(\mathcal{J}_{m}) & =  S(J_q \otimes \mathcal{A}_{m-1}) \\
 & = S( J_q \otimes (I_q \otimes \mathcal{J}_{m-2} + Q \otimes \mathcal{A}_{m-2})) \\
 & = S((J_q \otimes I_q) \otimes \mathcal{J}_{m-2}) + S((J_q \otimes Q) \otimes \mathcal{A}_{m-2}) \\
 & = S(J_q)S(I_q)S(\mathcal{J}_{m-2}) + S(J_q)S(Q)S(\mathcal{A}_{m-2}) \\
 & = q^3\, S(\mathcal{J}_{m-2})
\end{split}
\label{eq:recJ}
\end{equation} 
and
\begin{equation}
\begin{split}
 S(\mathcal{A}_{m}) & =  S(I_q \otimes \mathcal{J}_{m-1} + Q \otimes \mathcal{A}_{m-1}) \\
 & = S( I_q \otimes (J_q \otimes \mathcal{A}_{m-2}) ) + S(Q)S(\mathcal{A}_{m-1}) \\
 & = S(J_q)S(I_q)S(\mathcal{A}_{m-2}) \\
 & = q^3 \, S(\mathcal{A}_{m-2}).
\end{split}
\label{eq:recA}
\end{equation}
We now prove that $S(\mathcal{J}^{(q)}_{2m}) = S(\mathcal{A}^{(q)}_{2m}) = q^{3m}$ by induction on $m$. For the base case, notice $S(\mathcal{J}_0) = S(\mathcal{A}_0) = 1$. Now suppose $k \ge 1$ and that $S(\mathcal{J}^{(q)}_{2k}) = S(\mathcal{A}^{(q)}_{2k}) = q^{3k}$. Equation~(\ref{eq:recJ}) together with the induction hypothesis implies $S(\mathcal{J}_{2(m+1)}) = q^3 \, S(\mathcal{J}_{2m}) = q^{3(k+1)}$. Similarly, Equation~(\ref{eq:recA}) and the induction hypothesis imply $S(\mathcal{A}_{2(m+1)}) = q^{3(m+1)}$. Thus for all positive integers $m$ we have $S(\mathcal{J}^{(q)}_{2m}) = S(\mathcal{A}^{(q)}_{2m}) = q^{3m}$. It follows that (i) holds. We can prove (ii) and (iii) similarly.
\end{proof}

Lemma~\ref{lem:excessJA} makes it easy to calculate the excess of the $\FUH{q^m}{q}$'s in Equation~(\ref{eqn:rec4UH}). To compute the excess of these matrices, we consider separately the cases when $m$ is odd and even. First, use Lemma~\ref{lem:excessJA} to observe that
	\begin{align*}
		S\left( \frac{1}{\sqrt{q+1}} \Jq_{2m} + i \sqrt{\frac{q}{q+1}} \Aq_{2m}\right)
		&= \frac{1}{\sqrt{q+1}} S\left(\J_{2m}\right) + i \sqrt{\frac{q}{q+1}} S\left(\A_{2m}\right) \\
		&= \frac{q^{3m}}{\sqrt{q+1}}\left( 1 + i \sqrt{q} \right).
	\end{align*}
	Therefore,
	$$ \left|S\left( \frac{1}{\sqrt{q+1}} \Jq_{2m} + i \sqrt{\frac{q}{q+1}} \Aq_{2m}\right)\right| = \left|\frac{q^{3m}}{\sqrt{q+1}}\left( 1 + i \sqrt{q} \right)\right| = q^\frac{3(2m)}{2}. $$
	Using a similar computation one can show that
	$$ \left|S\left( \frac{1}{\sqrt{q+1}} \Jq_{2m+1} + i \sqrt{\frac{q}{q+1}} \Aq_{2m+1}\right)\right| = q^\frac{3(2m+1)}{2}. $$
	Therefore, for any $m \ge 0$ we have
	$$ \left|S \left( \frac{1}{\sqrt{q+1}} \Jq_{m} + i \sqrt{\frac{q}{q+1}} \Aq_{m}\right)\right| = q^\frac{3m}{2}.$$
The excess meets Best's upper bound \cite{Best}, so the matrices are regular. 

In summary, we have established the following theorem.

\begin{thm}
Let $q$ be the order of a skew symmetric core. Then for each positive integer $m$, there is a regular $\FUH{q^m}{q}$ with excess $q^\frac{3m}{2}$. Furthermore, if $q$ is an odd prime power $q \equiv 3 \pmod{4}$ the constructed regular quaternary unit Hadamard matrix is multicirculant.
\end{thm}

\section{Association schemes}
There are many relationships between Hadamard matrices and association schemes; see \cite[Theorem 1.8.1]{BCN}, \cite{GC,HT}  for Hadamard matrices and \cite{cg10,IM} for unit Hadamard matrices. 
In this section we show that the quaternary unit Hadamard matrices are contained in some commutative  association scheme. 

A \emph{(commutative) association scheme of class $d$}
with vertex set $X$ of size $n$ 
is a set of non-zero $(0,1)$-matrices $A_0, \ldots, A_d$, which are called {\em adjacency matrices}, with
rows and columns indexed by $X$, such that:
\begin{enumerate}[(i)]
\item $A_0=I_n$.
\item $\sum_{i=0}^d A_i = J_n$.
\item For any $i\in\{0,1,\ldots,d\}$, $A_i^T\in\{A_0,A_1,\ldots,A_d\}$.
\item For any $i,j\in\{0,1,\ldots,d\}$, $A_iA_j=\sum_{k=0}^d p_{ij}^k A_k$
for some $p_{ij}^k$'s.
\item For any $i,j\in\{0,1,\ldots,d\}$, $A_iA_j=A_jA_i$. 
\end{enumerate}
%The association scheme is called to be \emph{symmetric} if all $A_i$ are symmetric, \emph{nonsymmetric} otherwise.
%The \emph{intersection matrix} $B_i$ ($i\in\{0,1,\ldots,d\}$) is defined as follows: $B_i=(p_{ij}^k)_{j,k=0}^{d}$. 

The vector space spanned by the $A_i$'s forms a commutative algebra, denoted by $\mathcal{A}$ and is called the \emph{Bose-Mesner algebra} or \emph{adjacency algebra}.~There exists a basis of $\mathcal{A}$ consisting of the primitive idempotents, say $E_0=(1/n)J_n,E_1,\ldots,E_d$. 
Since  $\{A_0,A_1,\ldots,A_d\}$ and $\{E_0,E_1,\ldots,E_d\}$ are two bases of $\mathcal{A}$, there exist a change-of-bases matrix $P=(P_{ij})_{i,j=0}^d$ %, $Q=(Q_{ij})_{i,j=0}^d$
 so that
\begin{align*}
A_j=\sum_{i=0}^d P_{ij}E_i.%,\quad E_j=\frac{1}{n}\sum_{i=0}^d Q_{ij}A_i.
\end{align*}
The matrix $P$ %($Q$ respectively) 
is called to be the {\em %first %(second respectively) 
eigenmatrix}.

Write $Q=A_1-A_2$ for disjoint $(0,1)$-matrices $A_1,A_2$, and let $A_0=I_q$.  
Note that $A_1,A_2$ are the adjacency matrices of the doubly regular tournaments on $q$ vertices; see \cite{RB}.  
Let $\mathfrak{X}^{(q)}$ be the association scheme with adjacency matrices $A_0,A_1,A_2$. 
Then the association scheme has the following eigenmatrix $P$:
\begin{align*}
P&=\left(
\begin{array}{cccccc}
 1 & \frac{q-1}{2} & \frac{q-1}{2} \\
 1 & \frac{-1+\sqrt{-q}}{2} & \frac{-1-\sqrt{-q}}{2} \\
 1 & \frac{-1-\sqrt{-q}}{2} & \frac{-1+\sqrt{-q}}{2}
\end{array}
\right).
\end{align*}

We define $\mathfrak{X}_m^{(q)}$ as the association scheme obtained from the $m$-times tensor products of the adjacency matrices $\mathfrak{X}^{(q)}$.
The adjacency matrices of $\mathfrak{X}_m^{(q)}$ are $A_{i_1}\otimes\cdots\otimes A_{i_m}$, $(i_1,\ldots,i_m)\in\{0,1,2\}^m$. 
Letting $E_0,E_1,E_2$ be the primitive idempotents of $\mathfrak{X}^{(q)}$, the primitive idempotents of $\mathfrak{X}_m^{(q)}$ are  $E_{i_1}\otimes\cdots\otimes E_{i_m}$, $(i_1,\ldots,i_m)\in\{0,1,2\}^m$. 
Note that for suitable ordering of the indices of the adjacency matrices and the primitive idempotents, the eigenmatrix $P_m$ of $\mathfrak{X}_m^{(q)}$ is 
\begin{align*}
P_m=P\otimes\cdots\otimes P \quad \text{($m$ factors)}. 
\end{align*}
Now we have the following proposition.
\begin{prop}
The quaternary unit Hadamard matrix $\frac{1}{\sqrt{q+1}}\mathcal{J}_m^{(q)}+i\sqrt{\frac{q}{q+1}}\mathcal{A}_m^{(q)}$ is in the Bose-Mesner algebra of $\mathfrak{X}_m^{(q)}$. 
\end{prop}
\begin{proof}
We prove by induction that $\mathcal{J}_m^{(q)}$ and $\mathcal{A}_m^{(q)}$ are in the Bose-Mesner algebra of $\mathfrak{X}_m^{(q)}$. 
The cases for $m=1,2$ are clear. 

Assume the cases $m-1,m-2$ to be true. 
The Bose-Mesner algebra of $\mathfrak{X}_m^{(q)}$ contains elements $J_q\otimes A_{i_2}\otimes\cdots \otimes A_{i_m}$ where $A_{i_2}\otimes\cdots \otimes A_{i_m}$ is an adjacency matrix of $\mathfrak{X}_{m-1}^{(q)}$. 
Using the induction hypothesis for $m-1$, the Bose-Mesner algebra of $\mathfrak{X}_m^{(q)}$ contains $\mathcal{J}_m^{(q)}=J_q\otimes \mathcal{A}_{m-1}^{(q)}$. 
By 
\begin{align*}
I_q\otimes \mathcal{J}_{m-1}^{(q)}&=I_q\otimes J_q \otimes \mathcal{A}_{m-2}^{(q)}, \\
Q\otimes \mathcal{A}_{m-1}^{(q)}&=(A_1-A_2)\otimes \mathcal{A}_{m-1}^{(q)},
\end{align*}
and the induction hypothesis for $m-1$ and $m-2$,   the Bose-Mesner algebra of $\mathfrak{X}_m^{(q)}$ contains $\mathcal{A}_m^{(q)}=I_q\otimes \mathcal{J}_{m-1}^{(q)}+Q\otimes  \mathcal{A}_{m-1}^{(q)}$.
This completes the proof. 
\end{proof}
\noindent {\bf Acknowledgments.}
Part of this note in contained in Kai Fender's undergraduate Honour's thesis written under supervision of Hadi Kharaghani. Hadi Kharaghani is supported in part by an NSERC Discovery Grant and ULRF.  Sho Suda is supported by JSPS KAKENHI Grant Number 15K21075. Kai Fender was supported by an NSERC-USRA. 

Thanks to the anonymous referees for their invaluable suggestions. The short and elegant proof of Lemma 2.2 is due to a referee which replaces our longer proof.

\bibliographystyle{plain}
\bibliography{4v-as-2017-refs} % References come BEFORE appendices.

\end{document}